\documentclass[final,1p,number,fleqn]{elsarticle}
\usepackage{amsmath,amssymb,amsthm,mathptmx,mathrsfs}

\DeclareMathOperator{\RE}{Re} 
\numberwithin{equation}{section}
\newtheorem{theorem}{Theorem}[section]
\newtheorem{lemma}{Lemma}[section]
\newtheorem{corollary}{Corollary}[section]
\theoremstyle{remark}
\newtheorem{rem}{Remark}[section]
\newtheorem{example}{Example}[section]

\newcommand{\blm}{\begin{lemma}\em }
\newcommand{\elm}{\end{lemma}}
\newtheorem{defi}{Definition}[section]
\newcommand{\brdef}{\begin{defi}}
\newcommand{\erdef}{\end{defi}}
\newcommand{\bcor}{\begin{corollary}\em }
\newcommand{\ecor}{\end{corollary}}
\newcommand{\bproof}{\begin{proof}}
\newcommand{\eproof}{\end{proof}}
\newcommand{\beqa}{\begin{equation}}
\newcommand{\eeqa}{\end{equation}}
\newcommand{\beq}{\begin{equation*}}
\newcommand{\eeq}{\end{equation*}}

 \newcommand{\set}[1]{\left\{#1\right\}}
 \newcommand{\sqb}[1]{\left[#1\right]}

 \newcommand{\ds}{H_p^{l,m}[\alpha_1]f(z)}
 \newcommand{\dsa}{H_p^{l,m}[\alpha_1 +1]f(z)}
 
 \newcommand{\mt}{I_p(r,\lambda)f(z)}
 \newcommand{\mta}{I_p(r+1,\lambda)f(z)}

 \newcommand{\jo}{J_\kappa^{l,m}[\alpha_1]f(z)}

\makeatletter
\newcommand{\subjclassname@later}{\textup{2010} Mathematics Subject
Classification} \makeatother

\journal{Computers \& Mathematics with Applications}

\begin{document}

\begin{frontmatter}
\title{Subordination and superordination for multivalent functions\\ defined by linear
operators}

\author[ssp]{S. Sivaprasad Kumar\corref{cor1}}
\address[ssp]{Department of Applied Mathematics,
Delhi Technological University, Delhi-110042, India}
\ead{spkumar@dce.ac.in}

\author[ssp]{Virendra Kumar}
\ead{vk.tripathi1@yahoo.com}

\author[vr,vr2]{V. Ravichandran }
\address[vr]{Department of Mathematics, University of Delhi,
Delhi 110 007, India} \ead{vravi@maths.du.ac.in}
\address[vr2]{School of Mathematical Sciences, Universiti Sains Malaysia,
11800 USM, Penang, Malaysia}


\cortext[cor1]{Corresponding author}

\begin{abstract}
In this paper, certain linear operators defined on $p$-valent
analytic functions have been unified and for them some subordination
and superordination results as well as the corresponding sandwich
type results  are obtained.  A related integral transform is
discussed and sufficient conditions for functions in different
classes have been obtained.
\end{abstract}


\end{frontmatter}


\section{Introduction}
Let $\mathcal{H}$ be the class of functions analytic in
$\mathbb{U}:=\{z\in\mathbb{C}:|z|<1\}$ and $\mathcal{H}[a,n]$ be the
subclass of $\mathcal{ H}$ consisting of functions of the form
$f(z)=a+a_nz^n+a_{n+1}z^{n+1}+\ldots$. Let $\mathcal{A}_p$ denote
the class of all analytic functions of the form \beqa
\label{eq1}f(z)=z^p+\sum_{k=p+1}^\infty a_kz^k \quad (z\in
\mathbb{U}) \eeqa and let  $\mathcal{A}_1:=\mathcal{A}.$  For two
functions $f(z)$ given by (\ref{eq1}) and $ g(z)=z^p +
\sum_{k=p+1}^{\infty}b_k z^k$, the Hadamard product (or convolution)
of $f$ and $g$ is defined by
\begin{equation} \label{e1.4}(f*g)(z):=z^p +
\sum_{k=p+1}^{\infty}a_k b_k z^k=:(g*f)(z).\end{equation} For two
analytic functions $f$ and $g$, we say that $f$ is
\emph{subordinate} to $g$ or $g$ \emph{superordinate} to $f$, if
there is a Schwarz function $w$ with $|w(z)|\leq |z|$ such that
$f(z)=g(w(z)).$ If $g$ is univalent, then $f\prec g$ if and only if
$f(0)=g(0)$ and $f(\mathbb{U})\subseteq g(\mathbb{U})$. The class
$R(\alpha)$ is defined by $$R(\alpha):=\set{f\in
\mathcal{A}:\RE\frac{f(z)}{z}>\alpha, 0\leq\alpha<1; z\in
\mathbb{U}}$$ and $R=R(0).$ The class $S^*(\alpha)$ of starlike
functions of order $\alpha$ is defined as
$$S^*(\alpha):=\set{f\in \mathcal{A}:\RE \frac{zf'(z)}{f(z)}>\alpha,0\leq\alpha<1; z\in \mathbb{U}}.$$
Note that $S^*(0)=S^*,$ the class of starlike functions. The class
of starlike functions of reciprocal order $\alpha$ is denoted by
$S_r^*(\alpha)$ and is given by
$$S_r^*(\alpha):=\set{f\in S^*:\RE \frac{f(z)}{zf'(z)}>\alpha, 0\leq\alpha<1; z\in \mathbb{U}}.$$
Note that $S_r^*(0)=S^*$. For $-1\leq B<A\leq1,$ Janowski \cite{jano} introduced the class
$S^*[A,B]$  given by
$$S^*[A,B]:=\set{f\in \mathcal{A}:\frac{zf'(z)}{f(z)}\prec\frac{1+Az}{1+Bz}; z\in \mathbb{U}}.$$ For $A=1$ and $B=-1$, it
reduces to the class $S^*$. A function $f\in \mathcal{A}$ is said to
be strongly starlike function of order $\eta$ if it satisfies
$$\left|\arg\frac{zf'(z)}{f(z)}\right|<\frac{\eta\pi}{2}\quad(0<\eta\leq1;
z\in \mathbb{U})$$ or equivalently
$$\frac{zf'(z)}{f(z)}\prec\left(\frac{1+z}{1-z}\right)^\eta\quad(0<\eta\leq1;
z\in \mathbb{U}).$$ The class of all such functions is denoted by
$SS^*(\eta)$. Obviously, $SS^*(1)=S^*$. The class
$\mathcal{SL}(\eta)$ is defined by
$$
\mathcal{SL}(\eta):=\left\{f\in
\mathcal{A}:\left|\left(\frac{zf'(z)}{f(z)}\right)^{\frac{1}{\eta}}-1\right|<1,\eta>0; \;z\in \mathbb{U}\right\}
$$ or equivalently
$$zf'(z)/f(z) \prec (1+z)^\eta\quad
(\eta>0; \;z\in \mathbb{U}).$$ Note that the class
$\mathcal{SL}:=\mathcal{SL}(\frac{1}{2})$, was introduced by
Sok\'o\l\ and Stankiewicz \cite{sokol96} and studied recently by Rosihan M. Ali\ et al., \cite{lemi}.

For $\alpha_j\in\mathbb{C}\quad (j=1,2,\ldots, l)$ and
$\beta_j\in\mathbb{C} \setminus \{0,-1,-2,\ldots\}\ (j=1,2,\ldots
m),$ the {\em generalized hypergeometric function} $_lF_m(\alpha_1,
\ldots, \alpha_l; \beta_1,\ldots,\beta_m; z)$ is defined by  the
infinite series
\[
_lF_m(\alpha_1, \ldots, \alpha_l; \beta_1,\ldots,\beta_m; z):=
\sum_{n=0}^\infty \frac{(\alpha_1)_n \ldots(\alpha_l)_n }{
(\beta_1)_n \ldots (\beta_m)_n } \frac{z^n}{n!}\]
\[ (l\leq m+1; l,m\in\mathbb{N}_0:=\{0,1,2,\ldots\})
\]
where $(a)_n$ is the Pochhammer symbol defined by
\[ (a)_n:=\frac{\Gamma(a+n)}{\Gamma(a)}=\left\{%
\begin{array}{ll}
    1, & \hbox{$(n=0)$;} \\
    a(a+1)(a+2)\dots(a+n-1), & \hbox{$(n\in\mathbb{N}:=\{1,2,3\ldots\})$.}
\end{array}%
\right.\] Corresponding to the function
\begin{equation}\label{hp}
h_p(\alpha_1, \ldots, \alpha_l; \beta_1,\ldots,\beta_m; z):= z^p\
_lF_m(\alpha_1, \ldots, \alpha_l; \beta_1,\ldots,\beta_m; z),
\end{equation}
the Dziok-Srivastava operator \cite{DHMS03} (see also \cite{HMS01})
$H_p^{(l,m)} (\alpha_1, \ldots, \alpha_l; \beta_1,\ldots,\beta_m)$
is defined by the Hadamard product
\begin{eqnarray}\label{DHMS}\nonumber
H_p^{(l,m)} (\alpha_1, \ldots, \alpha_l; \beta_1,\ldots,\beta_m)f(z)
&:= &
h_p(\alpha_1, \ldots, \alpha_l; \beta_1,\ldots,\beta_m; z)*f(z) \\
& = & z^p+\sum_{n=p+1}^\infty \frac{(\alpha_1)_{n-p}
\ldots(\alpha_l)_{n-p} }{ (\beta_1)_{n-p} \ldots (\beta_m)_{n-p} }
\frac{a_nz^n}{(n-p)!}.
\end{eqnarray}
For brevity, we write
\begin{equation}
 H_p^{l,m}[\alpha_1]:=H_p^{(l,m)}(\alpha_1, \ldots,\alpha_l; \beta_1,\ldots,\beta_m)
 \end{equation} and we have the following identity:
 \begin{equation} \label{r1}
    z(\ds)'=\alpha_1\dsa-(\alpha_1-p)\ds.
\end{equation}

Special cases of the Dziok-Srivastava linear operator includes the
Hohlov linear operator \cite{Hoh78}, the Carlson-Shaffer linear
operator \cite{CS84}, the Ruscheweyh derivative operator
 \cite{Rus75}, the generalized Bernardi-Libera-Livingston
 linear integral operator ({\em cf.} \cite{Ber69},
\cite{Lib65}, \cite{Liv66}) and the Srivastava-Owa fractional
derivative operators ({\em cf.} \cite{Owa78}, \cite{OwaHMS87}).\\

Motivated by the multiplier transformation on $\mathcal{A}$, we
define  the operator $I_p(r,\lambda)$ on $\mathcal{A}_p$ by the
following  infinite series
\begin{equation}\label{Mul}
 I_p(r,\lambda)f(z):= z^p+\sum_{n=p+1}^{\infty}
  \left( \frac{n+\lambda}{p+\lambda}\right)^r a_n z^n\quad (\lambda\in\mathbb{C} \setminus \{-1,-2,\ldots\})
\end{equation} and we have the following identity:
\begin{equation} \label{r4}
    z(\mt)'=(p+\lambda)\mta-\lambda\mt.
\end{equation}

For $\lambda\geq0$, the operator was introduced and studied by
Ravichandran and Sivaprasad Kumar \cite{siva} and extensively used by many
authors ({\em cf.} \cite{ravi}, \cite{ khar}, \cite{siva1}). The
operator $I_p(r,\lambda)$ is closely related to the S\v{a}l\v{a}gean
derivative operators \cite{Sal83}. The operator
$I^r_\lambda:=I_1(r,\lambda)$ was studied by Cho and
Srivastava \cite{CHMS03} and Cho and Kim \cite{CK03}. The operator
$I_r:=I_1(r,1)$ was studied by Uralegaddi and Somanatha \cite{US92}.

Corresponding to the function $h_p$ defined in (\ref{hp}),
 Al-Kharasani and Al-Areefi \cite{khar} introduced a function $F_\kappa(\alpha_1,\ldots, \alpha_l; \beta_1,\ldots,\beta_m; z)$ given by \[h_p(\alpha_1, \ldots,\alpha_l; \beta_1,\ldots,\beta_m; z)*
 F_\kappa(\alpha_1,\ldots, \alpha_l; \beta_1,\ldots,\beta_m; z)=\frac{z^p}{(1-z)^{\kappa+p-1}} \quad{(\kappa>0; z\in\mathbb{U})}\]
and defined a new linear operator $J_\kappa(\alpha_1,
\ldots,\alpha_l; \beta_1,\ldots,\beta_m; z)$, analogous to
$H_p^{l,m}[\alpha_1]$, by
\begin{equation}\label{eqj} J_\kappa(\alpha_1, \ldots,\alpha_l; \beta_1,\ldots,\beta_m; z)f(z)= F_\kappa(\alpha_1, \ldots,\alpha_l; \beta_1,\ldots,\beta_m; z)*f(z)\end{equation}
where $\alpha_j\in\mathbb{C}\quad (j=1,2,\ldots, l)$ and
$\beta_j\in\mathbb{C} \setminus \{0,-1,-2,\ldots\}\ (j=1,2,\ldots
m),\; z\in\mathbb{U}, \kappa>0.$ For convenience, we write
\begin{equation}
 J_\kappa^{l,m}[\alpha_1]:=J_\kappa(\alpha_1, \ldots,\alpha_l; \beta_1,\ldots,\beta_m).
 \end{equation}
 They established  the following identity:
 \begin{equation} \label{r2}
    z(\jo)'=(\alpha_1 -1)J_\kappa^{l,m}[\alpha_1-1]f(z)-(\alpha_1-p-1)\jo.
\end{equation}
 Special cases of this operator are when $p=1,$ it reduces to the operator defined in \cite{kwon}, when $p=1,\kappa=2$ it is the Noor's integral operator defined in \cite{noor}.
Now consider the following infinite series:
\begin{equation}\label{i1}\mathscr{F}_\lambda^r(z)=z^p+\sum_{n=p+1}^{\infty}
\left( \frac{n+\lambda}{p+\lambda}\right)^r z^n\quad
(\lambda\in\mathbb{C} \setminus \{-1,-2,\ldots\}),\end{equation} we
have
$$I_p(r,\lambda)f(z)=\mathscr{F}_\lambda^r(z)*f(z).$$
Corresponding to the function $\mathscr{F}_{\lambda,\kappa}^r(z)$
given by
$$\mathscr{F}_\lambda^r(z)*\mathscr{F}_{\lambda,\kappa}^r(z)=\frac{z^p}{(1-z)^{\kappa+p-1}}
\quad{(z\in\mathbb{U}; \; \kappa>0)}, $$ Al-Kharasani and Al-Areefi
\cite{khar} defined  the multiplier transform $T_\kappa(r,\lambda)$
as follows:
\begin{equation}\label{eqt} T_\kappa(r,\lambda)f(z)=\mathscr{F}_{\lambda,\kappa}^r(z)*f(z)\quad{(\lambda\in\mathbb{C} \setminus \{-1,-2,\ldots\}, \kappa>0; f\in \mathcal{A}_p, z\in\mathbb{U})}\end{equation} and established the identity
\begin{equation}\label{r5} z(T_\kappa(r,\lambda)f(z))'=(p+\lambda)T_\kappa(r-1,\lambda)f(z)-\lambda T_\kappa(r,\lambda)f(z).\end{equation}
 When $p=1$, this operator is a generalization of the linear operator defined in ~\cite{noor1}.
Recently Miller and Mocanu \cite{Miller2003} considered certain
second order differential superordinations. Using the results of
Miller and Mocanu \cite{Miller2003}, Bulboac\u a \cite{Bul2002b}
have considered certain classes of first order differential
superordinations and Bulboac\u a \cite{Bul2002} considered certain
superordination-preserving integral operators. Later many papers in
this direction emerged ({\em cf.} \cite{ravi}, \cite{khar},
\cite{DHMS03}, \cite{siva}, \cite{siva1}, \cite{HMS01}).

Jung, Kim and Srivastava\label{jks} introduced the linear operator
on $\mathcal{A}$ is defined by
\begin{displaymath}Q^{\alpha}_{\beta}(f)={\alpha+\beta \choose \beta}
\frac{\alpha}{z^\beta}\int_{0}^{z}\left(1-\frac{t}{z}\right)^{\alpha-1}t^{\beta-1}f(t)dt,
\quad(\alpha\geq0, \beta>-1, f\in\mathcal{A}).\end{displaymath}

Note that \begin{equation}\label{eqjks}
Q^{\alpha}_{\beta}(f)=z+\sum_{n=2}^{\infty}\frac{\Gamma(\beta+n)
\Gamma(\alpha+\beta+1)}{\Gamma(\alpha+\beta+n)\Gamma(\beta+1)}a_nz^n.\end{equation}
Motivated by the above linear operator introduced by Jung, Kim,
Srivastava \cite{jks}, Liu introduced the following integral operator
on $\mathcal{A}_p$ \cite{liu}:
$$Q^{\alpha}_{\beta, p}(f)=z^p+\sum_{n=p+1}^{\infty}\frac{\Gamma(\beta+n+p)
\Gamma(\alpha+\beta+p)}{\Gamma(\alpha+\beta+n+p)\Gamma(\beta+p)}a_nz^n
\quad(\alpha\geq0, \beta>-1, f\in\mathcal{A}_p).$$ Note that if
$$F^\alpha_\beta(z):=z^p+\sum_{n=p+1}^{\infty}\frac{\Gamma(\beta+n+p)
\Gamma(\alpha+\beta+p)}{\Gamma(\alpha+\beta+n+p)\Gamma(\beta+p)}
z^n,$$ then $$Q^{\alpha}_{\beta, p}(f)=F^\alpha_\beta(z)*f(z).$$
Further it can be shown that
\begin{equation}\label{r3} z[Q^{\alpha}_{\beta, p}(f)]'=(\alpha+\beta+p-1)Q^{\alpha-1}_{\beta, p}(f)-(\alpha+\beta-1)
Q^{\alpha}_{\beta, p}(f).\end{equation} Since certain important
properties of the classes defined by the above mentioned linear
operators essentially depend on the recurrence relation~(\ref{r1}),
(\ref{r4}), (\ref{r2}), (\ref{r5}) and (\ref{r3}). We define a class
of operators and a corresponding class of functions in the
following: \brdef Let $O_p$ be the class of all linear operators
$L_p^{a}$  defined on $\mathcal{A}_p$ satisfying
$$z[L_{ p}^{a}f(z)]'=\alpha_aL_p^{a+1}f(z)-(\alpha_a-p)L_p^{a}f(z).$$
One can also consider the class of linear operators satisfying
$$z[L_{ p}^{b}f(z)]'=\alpha_bL_p^{b-1}f(z)-(\alpha_b-p)L_p^{b}f(z).$$
\erdef However, in this paper, we restrict ourself to the first case  as
the results pertaining to the second class of operators are much akin to their counter parts in the first case.
 We note that if $L_p^k(f(z))=\mathscr{L}_k(z)*f(z)$, then $L_p^k$ unifies  the above  stated all operators for suitable function $\mathscr{L}_k(z)$ assumes as follows.

$$
L_p^{k} =\left\{
   \begin{array}{ll}
     H_p^{l,m}[\alpha_1], & \hbox{for  $\mathscr{L}_{\alpha_1}(z)=h_p(\alpha_1, \ldots, \alpha_l; \beta_1,\ldots,\beta_m; z),\;k=a=\alpha_1$  } \\
     I_p(r,\lambda), & \hbox{for  $\mathscr{L}_{r}(z)=z^p+\sum_{n=p+1}^{\infty}
  \left( \frac{n+\lambda}{p+\lambda}\right)^rz^n,\;k=a=r$ } \\
   J_\kappa^{l,m}[\alpha_1], & \hbox{for  $\mathscr{L}_{\alpha_1}(z)=F_\kappa(\alpha_1, \ldots,\alpha_l; \beta_1,\ldots,\beta_m; z),\;k=b=\alpha_1$ } \\
    T_\kappa(r,\lambda), & \hbox{for  $\mathscr{L}_{r}(z)=\mathscr{F}_{\lambda,\kappa}^r(z),\; k=b=r$ } \\
  Q_{\beta,p}^{\alpha}\;, & \hbox{for  $\mathscr{L}_{\alpha}(z)=F^\alpha_\beta(z),\; k=b=\alpha$.}
   \end{array}
 \right.$$
 Thus the operators $ H_p^{l,m}[\alpha_1]$, $I_p(r,\lambda)$, $J_\kappa^{l,m}[\alpha_1]$, $T_\kappa(r,\lambda)$ and $Q_{\beta,p}^{\alpha}$ are in the class $O_p$.

In the present investigation,  we unify certain linear operators defined on $p$-valent functions and for them some key results with subordination and superordination  leading to some sandwich results are obtained. A related integral transform is also discussed.  Further sufficient conditions for functions belonging to the classes $R,\; S^*,\;S^*_r,\;SS^*$ and $\mathcal{SL}$ have been obtained using our key results. Hence most of the earlier results in this direction becomes special cases to our results, for instance, the results of Al-Kharsani and Al-Areefi \cite{khar} become  special case to our main results when $\mu=1$ and $\nu=0$.\\

\section{Preliminaries}

In our present investigation, we need the following: \brdef
\cite[Definition 2, p.817]{Miller2003} Denote by $\mathcal{ Q}$, the
set of all functions $f(z)$ that are analytic and injective on
$\overline{\mathbb{U}}-E(f)$, where
\[ E(f)=\{\zeta \in\partial \mathbb{U}: \lim_{z\rightarrow \zeta}
f(z)=\infty \},\]  and are such that $f'(\zeta)\not=0$ for
$\zeta\in\partial \mathbb{U}-E(f)$. \erdef
\begin{lemma}[cf. Miller and Mocanu{\cite[Theorem 3.4h, p.132]{miller}}]
\label{tha} Let $\psi(z)$ be univalent in the unit disk $\mathbb{U}$
and let $\vartheta$ and $\varphi$ be analytic in a domain $D \supset
\psi(\mathbb{U})$ with $\varphi(w)\neq 0,$ when $w \in
\psi(\mathbb{U}).$ Set
\[Q(z):=z\psi'(z)\varphi(\psi(z)),\quad h(z):=\vartheta(\psi(z))+Q(z).\]
Suppose that
\begin{enumerate}
\item  $Q(z)$ is starlike univalent in $\mathbb{U}$ and \item
${\RE } \frac{zh'(z)}{Q(z)}>0$ for  $z\in \mathbb{U}.$
\end{enumerate}
If $q(z)$ is analytic in $\mathbb{U}$, with $q(0)=\psi(0),\;
q(\mathbb{U})\subset D$ and \begin{equation} \label{eq8}
    \vartheta(q(z))+zq'(z)\varphi(q(z))\prec
    \vartheta(\psi(z))+z\psi'(z)\varphi(\psi(z)),
\end{equation} then $q(z)\prec \psi(z)$ and $\psi(z)$ is the best
dominant.
\end{lemma}
\begin{lemma}{\cite[Corollary 3.2, p.289]{Bul2002b}}
\label{Bul-lem}
 Let $\psi(z)$ be univalent in the unit disk
$\mathbb{U}$ and $\vartheta$ and $\varphi$ be analytic in a domain
$D$ containing $\psi(\mathbb{U})$. Suppose that
\begin{enumerate}
\item $\RE \left[\vartheta'(\psi(z))/\varphi(\psi(z))\right]>0
 \text{ for } z\in \mathbb{U}$,
\item $Q(z):=z\psi'(z)\varphi(\psi(z))$ is starlike univalent in
$\mathbb{U}$.
\end{enumerate}
If $q(z)\in \mathcal{H}[\psi(0),1]\cap \mathcal{Q}$, with
$q(\mathbb{U})\subseteq D$, and
$\vartheta(q(z))+zq'(z)\varphi(q(z))$ is univalent in $\mathbb{U}$,
then
\begin{equation}
 \label{lem2s1}
\vartheta(\psi(z))+z\psi'(z)\varphi(\psi(z)) \prec
\vartheta(q(z))+zq'(z)\varphi(q(z)),
\end{equation}
implies  $\psi(z)\prec q(z)$  and $\psi(z)$ is the best subordinant.
\end{lemma}
\brdef
 Let $f\in \mathcal{A}_p$, we define the function $\Omega_{L,\mu,\nu}^a$ by
 $$\Omega_{L,\mu,\nu}^a(f(z))=\left(\frac{L_p^{a+1}f(z)}{z^p}\right)^\mu\left(\frac{z^p}{L_p^af(z)}\right)^\nu$$
where the powers are principal one,  $\mu$ and $\nu$ are real
numbers  such that they do not assume the value zero simultaneously.
For the sake of convenience, let us denote
$$\Omega_{L,\mu,\nu}^a(f(z),F(z)):=\frac{\Omega_{L,\mu,\nu}^a(f(z))}{\Omega_{L,\mu,\nu}^a(F(z))}.$$
\erdef
\section{Sandwich Results}
We begin with the following theorem.
\begin{theorem}\label{t1} Let $\psi$ be convex univalent in
$\mathbb{U}$ with $\psi(0)=1$. Let $\RE
\sqb{\alpha_{a+1}\mu-\alpha_a\nu}\geq0$, $\alpha_{a+1}\neq 0$ and
$f\in \mathcal{A}_p$. Assume that $\chi$ and $\Phi$ are respectively
defined by
\begin{equation}\label{chi}
\chi(z):=\frac{1}{\alpha_{a+1}}\sqb{(\alpha_{a+1}\mu-\alpha_a\nu){\psi(z)}+z\psi'(z)}
\end{equation}
and
\begin{equation}\label{t1phi}
  \Phi(z):=  \Omega_{L,\mu,\nu}^{a} (f(z))\Upsilon_L(z),
\end{equation}
where
\begin{displaymath}
 \Upsilon_L(z):= \mu\Omega_{L,1,1}^{a+1} (f(z))-\frac{\alpha_a \nu}{\alpha_{a+1}}\Omega_{L,1,1}^{a} (f(z)).
\end{displaymath}

\begin{enumerate}
\item[1.] If $\Phi(z)\prec\chi(z)$, then
$$\Omega_{L,\mu,\nu}^{a}(f(z))\prec \psi(z)$$ and $\psi(z)$ is the best dominant.

\item[2.] If $\chi(z)\prec\Phi(z)$,
 \begin{equation} \label{th1c1}
 0\not=\Omega_{L,\mu,\nu}^{a}(f(z))\in \mathcal{H}[1,1]\cap \mathcal{Q}\; and\;
\Phi(z) \text{ is univalent in } \mathbb{U},\end{equation}
\end{enumerate}
then  $$ \psi(z)\prec\Omega_{L,\mu,\nu}^{a} (f(z))$$ and $\psi(z)$
is the best subordinant.
\end{theorem}

\begin{proof} Define the function $q$ by \begin{equation}
\label{t1p1} q(z):=\Omega_{L,\mu,\nu}^{a} (f(z)),
\end{equation} where the branch of $q(z)$ is so chosen such that $q(0)=1$. Then $q(z)$ is analytic in $\mathbb{U}$.
 By a simple computation, we find from (\ref{t1p1}) that

\begin{eqnarray}\label{t1p2}
  \nonumber  \frac{zq'(z)}{q(z)}&=&\frac{z[\Omega_{L,\mu,\nu}^{a} (f(z))]'}{\Omega_{L,\mu,\nu}^{a} (f(z))}\\
    &=&\mu \frac{z(L_p^{a+1}f(z))'}{L_p^{a+1}f(z)}
    -\nu\frac{z(L_p^{a}f(z))'}{L_p^{a }f(z)}+p(\nu-\mu).
\end{eqnarray}
By making use of the  identity \begin{equation} \label{t1p3}
    z(L_p^{a }f(z))'=\alpha_aL_p^{a+1}f(z)-(\alpha_a-p)L_p^{a }f(z),
\end{equation}in (\ref{t1p2}), we have
\begin{multline}
    \label{t1p4}\Omega_{L,\mu,\nu}^{a} (f(z))\left(\mu\Omega_{L,1,1}^{a+1} (f(z))-\frac{\alpha_a \nu}{\alpha_{a+1}}\Omega_{L,1,1}^{a} (f(z))\right)=\frac{1}{\alpha_{a+1}}[(\alpha_{a+1}\mu-\alpha_a\nu){q(z)}+zq'(z)].
\end{multline}
In view of (\ref{t1p4}), the subordination $\Phi(z)\prec \chi(z)$
becomes
\[(\alpha_{a+1}\mu-\alpha_a\nu){q(z)}+zq'(z)\prec(\alpha_{a+1}\mu-\alpha_a\nu){\psi(z)}+z\psi'(z)\]
and this can be written as (\ref{eq8}), by defining
\[ \vartheta(w):=(\alpha_{a+1}\mu-\alpha_a\nu){w} \text{ and }
\varphi(w):=1.\]
 Note that $\varphi(w)\neq 0$ and $\vartheta(w),\ \varphi(w)$ are
analytic in $\mathbb{C}-\{0\}$. Set
\begin{eqnarray}\nonumber
Q(z)&:=&z\psi'(z)\\ \nonumber
  h(z) &:=& \vartheta(\psi(z))+Q(z)
 = (\alpha_{a+1}\mu-\alpha_a\nu){\psi(z)}+z\psi'(z).
\end{eqnarray}
In light of the hypothesis of our Theorem~\ref{t1}, we see that
$Q(z)$ is starlike and {\small
\[\RE \left(\frac{zh'(z)}{Q(z)}\right)
=\RE
\left(\alpha_{a+1}\mu-\alpha_a\nu+1+\frac{z\psi''(z)}{\psi'(z)}\right)>0.\]}
By an application of Lemma~\ref{tha}, we obtain that $q(z)\prec
\psi(z)$ or
\[\Omega_{L,\mu,\nu}^{a} (f(z))\prec
\psi(z).\] The second half of Theorem~\ref{t1}  follows by a similar
application of Lemma~\ref{Bul-lem}.
\end{proof}
Using Theorem~\ref{t1},  we obtain the following ``sandwich
result''.
\begin{corollary}\label{cor1.2}
Let $\psi_j$ ($j=1,2$) be convex univalent in $\mathbb{U}$ with
$\psi_j(0)=1 $. Assume that $\RE
\sqb{\alpha_{a+1}\mu-\alpha_a\nu}\geq0$ and $\Phi$ be as defined in
(\ref{t1phi}). Further assume that
 \begin{equation}\label{t1c1}\nonumber
\chi_j(z):=\frac{1}{\alpha_{a+1}}\sqb{(\alpha_{a+1}\mu-\alpha_a\nu){\psi_j(z)}+z\psi_j'(z)}.\end{equation}
If (\ref{th1c1}) holds and  $\chi_1(z)\prec \Phi(z)\prec\chi_2(z)$,
then $$\psi_1(z)\prec \Omega_{L,\mu,\nu}^{a} (f(z))\prec\psi_2(z).$$
\end{corollary}
\begin{theorem} \label{t2} Let $\psi$ be convex univalent in
$\mathbb{U}$ with $\psi(0)=1$ and $\alpha_a$ be a complex number.
Assume that $\RE(\mu\alpha_{a+1}-\nu\alpha_a)\geq0$
 and $f\in \mathcal{A}_p$. Define
the functions $F$, $\chi$ and $\Psi$ respectively by
\begin{equation} \label{e2.12}
F(z):=\frac{\alpha_a}{z^{\alpha_a-p}}\int^z_0
t^{\alpha_a-p-1}f(t)dt,
\end{equation}
\begin{equation}\label{t2e2}
 \chi(z):=(\mu\alpha_{a+1}-\nu\alpha_a)\psi(z)+z\psi'(z)\end{equation}and
\begin{equation}\label{t2psi}
\Psi(z):=\Omega_{L,\mu,\nu}^{a}
(F(z))\left[\mu\alpha_{a+1}\Omega_{L,1,0}^{a}
(f(z),F(z))-\nu\alpha_a\Omega_{L,0,-1}^{a} (f(z),F(z))\right].
\end{equation}

\begin{enumerate}
  \item[1.] If $\Psi(z)\prec \chi(z)$, then $$\Omega_{L,\mu,\nu}^{a} (F(z))\prec
\psi(z)$$ and $\psi(z)$ is the best dominant.
  \item[2.] If $\chi(z)\prec \Psi(z)$,
\begin{equation}\label{t2c2}
0\not=\Omega_{L,\mu,\nu}^{a} (F(z))\in \mathcal{H}[1,1]\cap
\mathcal{Q} \text{ and} \;\Psi(z) \text{ is univalent in }
\mathbb{U},\end{equation}
\end{enumerate}

then $$\psi(z)\prec\Omega_{L,\mu,\nu}^{a} (F(z))$$ and $\psi(z)$ is
the best subordinant.
\end{theorem}

\begin{proof} From the definition of $F$, we obtain that
\begin{eqnarray}
\alpha_a L_{p}^{a}(f(z)) &=& (\alpha_a-p)
L_{p}^{a}(F(z))+z(L_{p}^{a}(F(z)))'.\label{t2p1}
\end{eqnarray}
Define the function $q$ by
\begin{equation} \label{t2p2}
q(z):= \Omega_{L,\mu,\nu}^{a} (F(z)),
\end{equation}
where the branch of $q(z)$ is so chosen such that $q(0)=1$. Clearly
$q(z)$ is analytic in $\mathbb{U}$.  Using (\ref{t2p1}) and
(\ref{t2p2}), we have
\begin{equation}\label{t2p3}
\Omega_{L,\mu,\nu}^{a} (F(z))\left(\mu\alpha_{a+1}\Omega_{L,1,0}^{a}
(f(z),F(z))-\nu\alpha_a\Omega_{L,0,-1}^{a}
(f(z),F(z))\right)=(\mu\alpha_{a+1}-\nu\alpha_a)q(z)+zq'(z).
\end{equation}
Using(\ref{t2p3}), the subordination $\Psi(z)\prec \chi(z)$ becomes
\[(\mu\alpha_{a+1}-\nu\alpha_a)q(z)+{zq'(z)}
\prec(\mu\alpha_{a+1}-\nu\alpha_a)\psi(z)+{z\psi'(z)}
\]
and this can be written as (\ref{eq8}), by defining
\[ \vartheta(w):=(\mu\alpha_{a+1}-\nu\alpha_a)\psi(z)\text{ and }
\varphi(w):=1.\] Note that $\varphi(w)\neq 0$ and $\vartheta(w),\
\varphi(w)$ are analytic in $\mathbb{C}-\{0\}$. Set
\begin{eqnarray}\nonumber
Q(z)&:=&z\psi'(z)\\ \nonumber h(z) &:=& \vartheta(\psi(z))+Q(z)
 = (\mu\alpha_{a+1}-\nu\alpha_a){\psi(z)}+z\psi'(z).
\end{eqnarray}
In light of the assumption of our Theorem~\ref{t2}, we see that
$Q(z)$ is starlike and {\small
\[\RE \left(\frac{zh'(z)}{Q(z)}\right)
=\RE
\left(\mu\alpha_{a+1}-\nu\alpha_a+1+\frac{z\psi''(z)}{\psi'(z)}\right)>0.\]}
An application of Lemma~\ref{tha}, gives $q(z)\prec \psi(z)$ or
\[ \Omega_{L,\mu,\nu}^{a} (F(z))\prec
\psi(z).\] By an application of Lemma~\ref{Bul-lem} the proof of the
second half of Theorem~\ref{t2} follows at once.
\end{proof}
As a consequence of Theorem~\ref{t2}, we obtain the following
``sandwich result".
\begin{corollary}\label{t1c2}
Let $\psi_j$ $(j=1,2)$ be convex univalent in $\mathbb{U}$ with
$\psi_j(0)=1 $ and $\alpha_a $ be a complex number. Further assume
that $\RE (\mu\alpha_{a+1}-\nu\alpha_a)\geq0$ and $\Psi$ be as
defined in (\ref{t2psi}). If (\ref{t2c2}) holds and $\chi_1(z)\prec
\Psi(z)\prec \chi_2(z),$ then $$\psi_1(z)\prec \Omega_{L,\mu,\nu}^{a}
(F(z))\prec\psi_2(z),$$ where
\begin{equation}\label{t2c1}\nonumber
\chi_j(z):=(\mu\alpha_{a+1}-\nu\alpha_a)\psi_j(z)+ z\psi_j'(z)
\quad(j=1,2)\end{equation} and $F$ is defined by (\ref{e2.12}).
\end{corollary}

\begin{theorem} \label{t5} Let $\phi$ be analytic in $\mathbb{U}$ with $\phi(0)=1$ and $\alpha_a$ is independent of $a$.
If $f\in \mathcal{A}_p$, then
$$\Omega_{L,\mu,\nu}^{a}(f(z))\prec\phi(z) \Leftrightarrow
\Omega_{L,\mu,\nu}^{a+1}(F(z))\prec\phi(z).$$ Further
$$\phi(z)\prec\Omega_{L,\mu,\nu}^{a}(f(z)) \Leftrightarrow
\phi(z)\prec\Omega_{L,\mu,\nu}^{a+1}(F(z)), $$ where $F$ is defined
by (\ref{e2.12}).
\end{theorem}
\begin{proof} From the definition of $F$, we obtain
\begin{equation}\label{e2.17}
\alpha_af(z)=(\alpha_a-p) F(z)+zF'(z).
\end{equation}
 By convoluting
(\ref{e2.17})  with $\mathscr{L}_k(z)$ and using the fact that
$z(f*g)'(z)=f(z)*zg'(z)$, we obtain
\[  \alpha_aL_{ p}^{a}(f(z)) = (\alpha_a-p)L_{ p}^{a}(F(z))+z(L_{ p}^{a}(F(z)))'\]
and by using the identity
\begin{equation} z[L_{ p}^{a}(f(z))]'=\alpha_aL_p^{a+1}(f(z))-(\alpha_a-p)L_p^{a}(f(z)),
\end{equation} we get
\begin{equation}\label{e2.18}
L^a_p(f(z))=L^{a+1}_p(F(z)).
 \end{equation}
Since $\alpha_a$ is independent of $a$,  $\alpha_{a+1}=\alpha_a,$
we have
\begin{eqnarray} \label{e2.19}\nonumber
\alpha_a L^{a+1}_p(f(z)) & = & z( L^{a }_p(f(z)))' +(\alpha_a-p)L^{a
}_p(f(z))\\ \nonumber
 & = & z( L^{a+1 }_p(F(z)))' +(\alpha_a-p)L^{a+1 }_p(F(z))\\
 & = & \alpha_{a+1}L^{a+2 }_p(F(z)).
\end{eqnarray}
Therefore, from (\ref{e2.18}) and (\ref{e2.19}),  we have
 $$\Omega_{L,\mu,\nu}^{a+1}(F(z))=\Omega_{L,\mu,\nu}^{a}(f(z))$$
 and hence the result follows at once.
 \end{proof}
Now we will use Theorem~\ref{t5} to state the following ``sandwich
result".
\begin{corollary}\label{t5c1} Let $f\in \mathcal{A}_p$ and $\alpha_a$ is independent of $a$. Let $\phi_i\; (i=1,2)$ be analytic in $\mathbb{U}$ with $\phi_i(0)=1$
and $F$ is defined by (\ref{e2.12}). Then
 $$\phi_1(z)\prec\Omega_{L,\mu,\nu}^{a}(f(z))\prec\phi_2(z) $$ if and only if
$$ \phi_1(z)\prec\Omega_{L,\mu,\nu}^{a+1}(F(z))\prec\phi_2(z).$$
\end{corollary}

\section{Applications}

We begin with some interesting applications of subordination part of
Theorem~\ref{t1} for the case when $L=H, $ the Dziok Srivastava
Operator. Note that the subordination part of Theorem~\ref{t1}
holds even if we assume $$\RE
\set{1+\frac{z\psi''(z)}{\psi'(z)}}>\max\{0,\RE[\alpha_1(\nu-\mu)-\mu]\}$$
instead of  ``$\psi(z)$ is convex and $\RE
\sqb{\alpha_1(\mu-\nu)+\mu}\geq0$" and leads to the following
corollary to the first part of Theorem~\ref{t1} by taking
$\psi(z)=(1+Az)/(1+Bz)$.
\begin{corollary}\label{cor1} Let $-1<B < A\leq1$ and $\RE (u-vB)\geq |v-\bar{u}B|$ where $u=\alpha_1(\mu-\nu)+\mu+1$ and $v=[\alpha_1(\mu-\nu)+\mu-1]B$. If $f\in \mathcal{A}_p$ satisfies the subordination
\begin{multline*}
\Omega_{H,\mu,\nu}^{\alpha_1} (f(z))\left(\mu\Omega_{H,1,1}^{\alpha_1+1} (f(z))-\frac{\alpha_1 \nu}{\alpha_1 +1}\Omega_{H,1,1}^{\alpha_1} (f(z))\right)\\
\prec\frac{1}{\alpha_1+1}\left([\alpha_1(\mu-\nu)+\mu]\frac{1+Az}{1+Bz}+\frac{(A-B)z}{(1+Bz)^2}\right)
\quad{(\alpha_1\neq-1)},
\end{multline*}
 then
$$\Omega_{H,\mu,\nu}^{\alpha_1}(f(z))\prec\frac{1+Az}{1+Bz}$$ and $(1+Az)/(1+Bz)$ is the best dominant.
\end{corollary}
\begin{proof}
Let
\begin{equation}\label{cor1p1}
\psi(z)=\frac{1+Az}{1+Bz} \quad(-1<B<A\leq1),
\end{equation}
then clearly $\psi(z)$ is univalent and $\psi(0)=1.$ Upon
logarithmic differentiation of $\psi$ given by (\ref{cor1p1}), we
obtain that
\begin{equation}\label{cor1p3}
z\psi'(z)=\frac{(A-B)z}{(1+Bz)^2}.
\end{equation}
Another differentiation of (\ref{cor1p3}), yields
\begin{equation}
1+ \frac{z\psi''(z)}{\psi'(z)}=\frac{1-Bz}{1+Bz}.
\end{equation}
If $z=re^{i\theta},$ $0\leq r<1$, then we have
$$\RE \left(1+ \frac{z\psi''(z)}{\psi'(z)}\right)=\frac{1-B^2r^2}{1+B^2r^2+2Br\cos\theta}\geq 0.$$
Hence $\psi(z)$ is convex in $\mathbb{U}.$  Also it follows that
\begin{eqnarray*}
  [\alpha_1(\mu-\nu)+\mu]+1+\frac{z\psi''(z)}{\psi'(z)} &=&  \frac{[\alpha_1(\mu-\nu)+\mu+1]+[\alpha_1(\mu-\nu)+\mu-1]B z}{1+B z}\\
   &=& \frac{u+vz}{1+B z},
\end{eqnarray*} where $u=\alpha_1(\mu-\nu)+\mu+1$ and $v=\sqb{\alpha_1(\mu-\nu)+\mu-1}B.$ The function $w(z)=\frac{u+vz}{1+B z}$ maps $\mathbb{U}$ into the disk
$$ \left|w-\frac{\bar{u}-\bar{v}B}{1-B^2}\right|\leq \frac{|v-\bar{u}B|}{1-B^2}.$$
Which implies  that
$$\RE \left([\alpha_1(\mu-\nu)+\mu]+1+\frac{z\psi''(z)}{\psi'(z)}\right)\geq \frac{\RE (\bar{u}-\bar{v}B)-|v-\bar{u}B|}{1-B^2}\geq 0 $$
provided
$$\RE (\bar{u}-\bar{v}B)\geq|v-\bar{u}B| $$
or $$\RE (u- v B)\geq|v-\bar{u}B|.$$ Thus the result follows at once
by an application of the first part of Theorem~\ref{t1}.
\end{proof}
\begin{corollary}\label{c1.1} Let $0\leq\alpha<1$ and $\RE (\alpha_1(\mu-\nu)+\mu)\geq0$. If \begin{multline*}
\Omega_{H,\mu,\nu}^{\alpha_1} (f(z))\left(\mu\Omega_{H,1,1}^{\alpha_1+1} (f(z))-\frac{\alpha_1 \nu}{\alpha_1 +1}\Omega_{H,1,1}^{\alpha_1} (f(z))\right)\\
\prec\frac{1}{\alpha_1+1}\left((\alpha_1(\mu-\nu)+\mu)\frac{1+(1-2\alpha)z}{1-z}+\frac{2(1-\alpha)z}{(1-z)^2}\right)
\quad{(\alpha_1\neq-1)},
\end{multline*}
 then
$$\RE \Omega_{H,\mu,\nu}^{\alpha_1}(f(z))>\alpha.$$
\begin{proof} Let
\begin{equation}\label{cor1p2}
\psi(z)=\frac{1+(1-2\alpha)z}{1-z} \quad(0\leq\alpha<1),
\end{equation}
then obviously $\psi(z)$ is univalent and $\psi(0)=1$. By a simple
calculation, we have
\begin{equation}
1+ \frac{z\psi''(z)}{\psi'(z)}=\frac{1+z}{1-z},
\end{equation}
which clearly indicates that $\psi(z)$ is convex. If we assume
$\beta=\alpha_1(\mu-\nu)+\mu$ then by  hypothesis we have
$\RE\beta\geq0.$ So if we take \begin{align*} w(z) &=  \beta
+\frac{1+z}{1-z}=  \frac{(1+\beta)+(1-\beta) z}{1-z},
\end{align*} then $w(z)$ maps the unit disc $\mathbb{U}$ on to $\RE{w}>\RE \beta\geq 0$. The result now follows by an application of the subordination part of Theorem~\ref{t1}.
\end{proof}
\end{corollary}
Note that if $p=1, l=m+1$ and $\alpha_{i+1}=\beta_i\;
(i=1,2,...,m),$ then
 $H_1[1]f(z)= f(z), H_1[2]f(z)=zf'(z)$ and $H_1[3]f(z)=\frac{1}{2}z^2f''(z)+ zf'(z)$. Putting $\alpha=1$
 in Corollaries ~\ref{cor1} and ~\ref{c1.1}, we obtain the following corollaries  respectively.
\begin{corollary} \label{ex1} Let $-1< B<A\leq 1$. Let $\mu$ and $\nu$ satisfy $(u-vB)\geq |v-uB|$ where $u=2\mu-\nu+1$ and $v=(2\mu-\nu-1)B$. If $f\in \mathcal{A}$ and satisfies the subordination
$$(f'(z))^\mu\left(\frac{z}{f(z)}\right)^\nu\left(\mu\left(2+\frac{zf''(z)}{f'(z)}\right)-\nu\frac{zf'(z)}{f(z)}\right)
\prec(2\mu-\nu)\frac{1+Az}{1+Bz}+\frac{(A-B)z}{(1+Bz)^2},$$ then
$$(f'(z))^\mu\left(\frac{z}{f(z)}\right)^\nu\prec\frac{1+Az}{1+Bz}$$ and $(1+Az)/(1+Bz)$ is the best dominant.
\end{corollary}
\begin{corollary}\label{c1.11} Let $0\leq\alpha<1$ and $2\mu\geq \nu$. If $f\in \mathcal{A}$ and satisfies
$$\RE\left((f'(z))^\mu\left(\frac{z}{f(z)}\right)^\nu\left(\mu\left(2+\frac{zf''(z)}{f'(z)}\right)
-\nu\frac{zf'(z)}{f(z)}\right)\right)>\frac{2(2\mu-\nu)\alpha-(1-\alpha)}{2},$$
then
$$\RE\left((f'(z))^\mu\left(\frac{z}{f(z)}\right)^\nu\right)>\alpha.$$
\end{corollary}
\begin{proof} From Corollary~\ref{c1.1}, we see that $$(f'(z))^\mu\left(\frac{z}{f(z)}\right)^\nu\left(\mu\left(2+\frac{zf''(z)}{f'(z)}\right)
-\nu\frac{zf'(z)}{f(z)}\right)\prec(2\mu-\nu)\frac{1+(1-2\alpha)z}{1-z}+\frac{2(1-\alpha)z}{(1-z)^2}=:h(z).$$
We now investigate the image of $h(\mathbb{U})$. Assuming
$a=1-2\alpha$
 and $b=2\mu-\nu$, we have
$$h(z)=\frac{b+(1+a-b+ab)z-abz^2}{(1-z)^2},$$ where $h(0)=b$ and $h(-1)=[2b(1-a)-(1+a)]/4$. The boundary curve of the image of $h(\mathbb{U})$ is given by $h(e^{i \theta})= u(\theta)+iv(\theta)$, $-\pi<\theta<\pi$,
where \[ u(\theta) =
\frac{(1+a-b+ab)+(1-a)b\cos\theta}{2(\cos\theta-1)}\quad \text{ and
} \quad v(\theta)=\frac{(1+a)b\sin\theta}{2(1-\cos\theta)}.\] By
eliminating $\theta$,
 we obtain the equation of the boundary curve as
\begin{equation}\label{para2} v^2= - b^2(1+a)\left(u-\frac{2b(1-a)-(a+1)}{4}\right).
\end{equation}
Obviously (\ref{para2}) represents a parabola opening towards the
left, with the vertex at the point
$\left(\frac{2b(1-a)-(a+1)}{4},0\right)$ and negative real axis as
its axis. Hence  $h(\mathbb{U})$ is the exterior of the parabola
(\ref{para2}) which includes the right half plane
$$u>\frac{2b(1-a)-(a+1)}{4}.$$
Hence the result follows at once.
\end{proof}
By setting $\mu=\nu=1$ in Corollary~\ref{ex1}, we obtain the
following example.
\begin{example}  \label{ex2} Let $-1<B<A\leq 1.$  If $f\in \mathcal{A}$ and satisfies the subordination
$$\frac{zf'(z)}{f(z)}\left(2-\frac{zf'(z)}{f(z)}+\frac{zf''(z)}{f'(z)}\right)
\prec\frac{1+Az}{1+Bz}+\frac{(A-B)z}{(1+Bz)^2},$$ then $f\in
S^*[A,B]$.
\end{example}
Putting $\mu=\nu=1$ in Corollary~\ref{c1.11}, we have the following
example.
\begin{example}\label{c1.14} Let $0\leq\alpha<1.$  If $f\in \mathcal{A}$ and satisfies
$$\RE\left(\frac{zf'(z)}{f(z)}\left(2-\frac{zf'(z)}{f(z)}+\frac{zf''(z)}{f'(z)}\right)\right)
>\frac{3\alpha-1}{2},$$
then $f\in S^*(\alpha)$.\end{example} If we take $\mu=1$ and $\nu=0$
in the Corollary~\ref{c1.11}, we have the following result.
\begin{example}\label{u1} Let $0\leq\alpha<1$. If $f\in \mathcal{A}$ and satisfies
$$\RE(2f'(z)+zf''(z))>\frac{5\alpha-1}{2}
,$$ then $\RE f'(z)>\alpha.$
\end{example}
\begin{rem}  Example~\ref{u1} provides  sufficient condition for univalence of $f(z)$ by Noshiro-Warschawski Theorem \cite[p.47]{nash}.
\end{rem}
Setting $\mu=0$ and $\nu=-1$ in Corollary~\ref{c1.11}, we obtain the
following result.
\begin{example}\label{ex1.4} Let $0\leq \alpha<1$. If $f\in \mathcal{A}$
and $\RE f'(z)>\frac{3\alpha-1}{2},$ then $f(z)\in R(\alpha).$
\end{example}
\begin{rem} When $\alpha=1/3$, the above result reduces to {\cite[Theorem 2]{obra}}.
\end{rem}
If we take $\psi(z)=((1+z)/(1-z))^\eta$ with $0<\eta\leq1$ in
Theorem~\ref{t1} for the case $L=H$, the Dziok Srivastava operator,
then clearly $\psi(z)$ is convex in $\mathbb{U}$ and consequently
corresponding to the subordination part of the Theorem~\ref{t1} we
have the following result.
\begin{corollary}\label{cor1.1} Let $0<\eta\leq1$, $\alpha_1\neq-1$ and $\RE ({\alpha_1(\mu-\nu)+\mu})\geq0$. If $f\in\mathcal{A}_p$
and satisfies \begin{multline*}
\Omega_{H,\mu,\nu}^{\alpha_1} (f(z))\left(\mu\Omega_{H,1,1}^{\alpha_1+1} (f(z))-\frac{\alpha_1 \nu}{\alpha_1 +1}\Omega_{H,1,1}^{\alpha_1} (f(z))\right)\\
\prec\frac{1}{\alpha_1+1}\left((\alpha_1(\mu-\nu)+\mu)+ \frac{2\eta
z}{1-z^2}\right)\left(\frac{1+z}{1-z}\right)^\eta,
\end{multline*}
 then
$$\Omega_{H,\mu,\nu}^{\alpha_1}(f(z))\prec\left(\frac{1+z}{1-z}\right)^\eta$$
and $((1+z)/(1-z))^\eta$ is the best dominant.
\end{corollary}
By taking $p=1, l=m+1,\alpha_1=1$ and $\alpha_{i+1}=\beta_i\;
(i=1,2,...m)$, in the above Corollary~\ref{cor1.1}, we have the
following result.
\begin{corollary}  \label{ex1.1}  Let $0<\eta\leq1$ and $2\mu\geq\nu$. If $f\in \mathcal{A}$ and satisfies \begin{displaymath}
\left|\arg\set{(f'(z))^\mu\left(\frac{z}{f(z)}\right)^\nu\left(\mu\left(2+\frac{zf''(z)}{f'(z)}\right)-
\nu\frac{zf'(z)}{f(z)}\right)}\right|< \frac{\delta \pi}{2},
\end{displaymath}
 then
$$\left|\arg\set{(f'(z))^\mu\left(\frac{z}{f(z)}\right)^\nu}\right|<\frac{\eta \pi}{2}$$
where
$$\delta=\eta+1-\frac{2}{\pi}\arctan\frac{2\mu-\nu}{\eta}.$$
\end{corollary}
\begin{proof} In the view of the Corollary~\ref{cor1.1}, we have
$$(f'(z))^\mu\left(\frac{z}{f(z)}\right)^\nu\left(\mu\left(2+\frac{zf''(z)}{f'(z)}\right)-
\nu\frac{zf'(z)}{f(z)}\right)\prec \left((2\mu-\nu)+ \frac{2\eta
z}{1-z^2}\right)\left(\frac{1+z}{1-z}\right)^\eta=:h(z)$$ implies
$$(f'(z))^\mu\left(\frac{z}{f(z)}\right)^\nu\prec\left(\frac{1+z}{1-z}\right)^\eta$$
or
$$\left|\arg\set{(f'(z))^\mu\left(\frac{z}{f(z)}\right)^\nu}\right|<\frac{\eta \pi}{2}\quad(z\in \mathbb{U}).$$
Now we need to find the minimum value of $\arg h(\mathbb{U})$. Let
$z=e^{i\theta}$. Since $h(\mathbb{U})$ is symmetrical about the real
axis, we shall restrict ourself to $0<\theta\leq\pi$. Setting
$t=\cot \theta/2$, we have $t\geq0$ and for  $z=\frac{it-1}{it+1}$,
we arrive at
\begin{eqnarray*}
h(e^{i\theta}) &=& (it)^{\eta-1}\left[(2\mu-\nu)it-\frac{\eta(1+t^2)}{2}\right]\\
 &=&(it)^{\eta-1}G(t),
 \end{eqnarray*}
 where $$G(t)=\left[(2\mu-\nu)it-\frac{\eta(1+t^2)}{2}\right].$$
Let $G(t)=U(t)+iV(t)$, where $U(t)=-\frac{\eta(1+t^2)}{2}$ and
$V(t)=(2\mu-\nu)t$, there arises two cases namely $2\mu>\nu$ and
$2\mu=\nu$. If $2\mu>\nu$, then a calculation
 shows that $\min_{t\geq0}\arg G(t)$ occurs at $t=1$ and
 $$\min_{t\geq0}\arg G(t)= \pi-\arctan\frac{2\mu-\nu}{\eta}.$$ Thus
 \begin{eqnarray*}\min_{|z|<1}\arg h(z)&=&\frac{(\eta+1)\pi}{2}-\arctan\frac{2\mu-\nu}{\eta}.
 \end{eqnarray*}
 If $2\mu=\nu$, then $\arg G(t)=\pi$ and $\min_{|z|<1}\arg h(z)=(\eta+1)\pi/2$. Thus for $2\mu\geq\nu$, we have
 \begin{eqnarray*} \min_{|z|<1}\arg h(z)&=& \min\set{\frac{(\eta+1)\pi}{2}, \frac{(\eta+1)\pi}{2}-\arctan{\frac{2\mu-\nu}{\eta}}}\\
 &=& \frac{(\eta+1)\pi}{2}-\arctan\frac{2\mu-\nu}{\eta}.
 \end{eqnarray*}
This completes the proof of the corollary.
 \end{proof}
By taking $\mu=\nu=1$ in the above Corollary~\ref{ex1.1},  we obtain
the following example.
\begin{example} Let $0<\eta\leq1$. If $f\in \mathcal{A}$ and satisfies
$$\left|\arg\set{\frac{zf'(z)}{f(z)}\left(2-\frac{zf'(z)}{f(z)}+\frac{zf''(z)}{f'(z)}\right)}\right|
<\frac{(\eta+1)\pi}{2}-\arctan\frac{1}{\eta},$$ then $f\in
SS^*(\eta)$.
\end{example}
By setting $\mu=1$ and $\nu=0$ in the Corollary~\ref{ex1.1},  we
have the following example.
\begin{example}\label{u2} Let $0<\eta\leq1$. If $f\in \mathcal{A}$ and satisfies
$$\left|\arg\set{ f'(z)\left(2+\frac{zf''(z)}{f'(z)}\right)}\right|<\frac{(\eta+1)\pi}{2}-\arctan\frac{2}{\eta}
,$$ then $|\arg f'(z)|<\frac{\eta \pi}{2}.$
\end{example}
By taking $\mu=0$ and $\nu=-1$ in Corollary~\ref{ex1.1},  we get the
following example.
\begin{example}\label{ex2.1} Let $0<\eta\leq1$. If $f\in \mathcal{A}$ and satisfies
$$\left|\arg f'(z)\right|<\frac{(\eta+1)\pi}{2}-\arctan\frac{1}{\eta},$$ then
$\left|\arg \frac{f(z)}{z}\right|<\frac{\eta \pi}{2}.$
\end{example}
We now enlist a few applications of Theorem~\ref{t1} for the
operator $L=H$, the Dziok Srivastava operator, by taking
$\psi(z)=\sqrt{1+z}$ as dominant. Obviously $\psi(z)$ is a convex
function in the open unit disk $\mathbb{U}$ with $\psi(0)=1$. The
subordination part of Theorem~\ref{t1}, leads to the following
result.
\begin{corollary}\label{lm1} Let $\alpha_1\neq-1$ and $\RE \sqb{{\alpha_1(\mu-\nu)+\mu}}\geq0.$ If $f\in\mathcal{A}_{p}$ and satisfies the subordination

\begin{multline*} \Omega_{H,\mu,\nu}^{\alpha_1} (f(z))\left(\mu\Omega_{H,1,1}^{\alpha_1+1} (f(z))-\frac{\alpha_1 \nu}{\alpha_1 +1}\Omega_{H,1,1}^{\alpha_1} (f(z))\right)\\
\prec
\frac{1}{\alpha_1+1}\left([\alpha_1(\mu-\nu)+\mu]\sqrt{1+z}+\frac{z}{2\sqrt{1+z}}\right),
\end{multline*}

  then
$$\Omega_{H,\mu,\nu}^{\alpha_1}(f(z))\prec\sqrt{1+z}$$
and $\sqrt{1+z}$ is the best dominant.
\end{corollary}
By taking $p=1, l=m+1$, $\alpha_1=1$ and $\alpha_{i+1}=\beta_i\;
(i=1,2,...m)$ in Corollary~\ref{lm1}, we obtain the following
result.
\begin{corollary}\label{lm1.1} Let $2\mu\geq\nu.$ If $f\in\mathcal{A}$ and satisfies the subordination
$$(f'(z))^\mu\left(\frac{z}{f(z)}\right)^\nu\left(\mu\left(2+\frac{zf''(z)}{f'(z)}\right)
-\nu\frac{zf'(z)}{f(z)}\right)\prec(2\mu-\nu)\sqrt{1+z}+\frac{z}{2\sqrt{1+z}},$$
then $$(f'(z))^\mu\left(\frac{z}{f(z)}\right)^\nu\prec\sqrt{1+z}$$
and $\sqrt{1+z}$ is the best dominant.
\end{corollary}
We obtain the following example from Corollary~\ref{lm1.1}.
\begin{example}\label{elm1.1} If $f\in\mathcal{A}$ and satisfies
$$\left|\frac{zf'(z)}{f(z)}\left(2+\frac{zf''(z)}{f'(z)}
-\frac{zf'(z)}{f(z)}\right)\right|<\sqrt{1.22}\approx 1.10,$$ then
$f\in\mathcal{SL}$.
\end{example}
\begin{proof} Putting $\mu=\nu=1$ in Corollary~\ref{lm1.1}, we have
$$\frac{zf'(z)}{f(z)}\left(2+\frac{zf''(z)}{f'(z)}
-\frac{zf'(z)}{f(z)}\right)\prec
\sqrt{1+z}+\frac{z}{2\sqrt{1+z}}=:h(z),$$ implies
$$\frac{zf'(z)}{f(z)}\prec\sqrt{1+z}.$$
The dominant $h(z)$ can be written as
$$h(z)=\frac{3z+2}{2\sqrt{1+z}}.$$
Writing $h(e^{i\theta})=u(e^{i\theta})+iv(e^{i\theta}),
-\pi<\theta<\pi$, we have
$$ u(\theta)= \frac{3\cos (3\theta/4)+2\cos(\theta/4)}{2\sqrt{2\cos(\theta/2)}}$$ and
$$v(\theta)= \frac{3\sin (3\theta/4)-2\sin(\theta/4)}{2\sqrt{2\cos(\theta/2)}}.$$ A simple calculation gives
$$ u^2(\theta)+v^2(\theta)=\frac{13+12\cos \theta}{8\cos(\theta/2)}=:k(\theta).$$
 A computation shows that $k(\theta)$ has minimum at
  $\theta=\arccos(\sqrt{1/24})$ and $k(\theta)\geq\sqrt{3/2}\approx1.22$. Since $h(0)=1$ and $h(-1)=-\infty$, by a computation we come to know that the image of $h(\mathbb{U})$ is the interior of the domain bounded by parabola opening towards left which contains the interior of the circle  $u^2+v^2=1.22$.
Hence the result follows at once.
\end{proof}
We now give some interesting applications of Theorem~\ref{t2} for
the case $L=H$. Note that if we replace the statement ``$\psi(z)$ is
convex in the open unit disc $\mathbb{U}$ and
 $\RE \sqb{(\mu-\nu)\alpha_1+\mu}\geq0$" by $$\RE \left(1+\frac{z\psi''(z)}{\psi'(z)}\right)>\max\{0,\RE \sqb{(\nu-\mu)\alpha_1-\mu}\}$$ in the hypothesis of Theorem~\ref{t2} still the subordination part of  the result holds so we obtain the following corollary as a straight forward  consequence to the first part of Theorem~\ref{t2} by taking $\psi(z)=(1+Az)/(1+Bz).$
\begin{corollary}\label{cor2} Let $-1<B<A\leq1$ and $\RE (u-vB)\geq |v-\bar{u}B|$ where $u=(\mu-\nu)\alpha_1+\mu+1$ and $v=[(\mu-\nu)\alpha_1+\mu-1]B$. If $f\in \mathcal{A}_p,$ $F$ as defined in (\ref{e2.12}) and
\begin{multline*}  \Omega_{H,\mu,\nu}^{\alpha_1} (F(z))(\mu(\alpha_1+1)\Omega_{H,1,0}^{\alpha_1}(f(z),F(z))-\nu\alpha_1\Omega_{H,0,-1}^{\alpha_1}(f(z),F(z)))\\
\prec
((\mu-\nu)\alpha_1+\mu)\frac{1+Az}{1+Bz}+\frac{(A-B)z}{(1+Bz)^2},\end{multline*}
then
$$\Omega_{H,\mu,\nu}^{\alpha_1}(F(z))\prec\frac{1+Az}{1+Bz}$$ and $(1+Az)/(1+Bz)$ is the best dominant.
\end{corollary}
\begin{corollary} \label{cor2.1} Let $0\leq\alpha<1$ and $\RE[(\mu-\nu)\alpha_1+\mu]\geq0$.
  If $f\in \mathcal{A}_p$, $F$ as defined in (\ref{e2.12}) and
 \begin{multline*}  \Omega_{H,\mu,\nu}^{\alpha_1} (F(z))(\mu(\alpha_1+1)\Omega_{H,1,0}^{\alpha_1}(f(z),F(z))-\nu\alpha_1\Omega_{H,0,-1}^{\alpha_1}(f(z),F(z)))
 \prec\\ ((\mu-\nu)\alpha_1+\mu) \frac{1+(1-2\alpha)z}{1-z}+\frac{2(1-\alpha)z}{(1-z)^2}
 \end{multline*} then
  $$ \Omega_{H,\mu,\nu}^{\alpha_1} (F(z))\prec\frac{1+(1-2\alpha)z}{1-z}$$
  and $(1+(1-2\alpha)z)/(1-z)$ is the best dominant.
\end{corollary}
Putting $p=1, l=m+1,\alpha_1=1$ and $\alpha_{i+1}=\beta_i\;
(i=1,2,...m)$ in Corollaries~\ref{cor2} and ~\ref{cor2.1}, we obtain
the following results respectively.
\begin{corollary} \label{ex3} Let $-1<B<A\leq1$ and $\RE (u-vB)\geq |v-\bar{u}B|$ where $u=2\mu-\nu+1$, $v=[2\mu-\nu-1]B$. If $f\in \mathcal{A}$, $F$ as defined in (\ref{e2.12}) and
$$(F'(z))^\mu\left(\frac{z}{F(z)}\right)^\nu\left(2\mu\frac{f'(z)}{F'(z)}-\nu\frac{f(z)}{F(z)}\right)
\prec(2\mu-\nu)\frac{1+Az}{1+Bz}+\frac{(A-B)z}{(1+Bz)^2},$$ then
$$(F'(z))^\mu\left(\frac{z}{F(z)}\right)^\nu\prec\frac{1+Az}{1+Bz}$$ and $(1+Az)/(1+Bz)$ is the best dominant.
\end{corollary}
\begin{corollary}\label{ex3.2} Let $0\leq\alpha<1$ and $2\mu\geq\nu$.
If $f\in \mathcal{A}$, $F$ as defined in (\ref{e2.12}) and
$$\RE\set{(F'(z))^\mu\left(\frac{z}{F(z)}\right)^\nu\left(2\mu\frac{f'(z)}{F'(z)}-\nu\frac{f(z)}{F(z)}\right)}
< \frac{2(2\mu-\nu)\alpha-(1-\alpha)}{2},$$ then
$$\RE\left[(F'(z))^\mu\left(\frac{z}{F(z)}\right)^\nu\right]>\alpha.$$
\end{corollary}
\begin{proof} From Corollary~\ref{cor2.1}, we see that $$(F'(z))^\mu\left(\frac{z}{F(z)}\right)^\nu\left(2\mu\frac{f'(z)}{F'(z)}-\nu\frac{f(z)}{F(z)}\right)
\prec(2\mu-\nu)\frac{1+(1-2\alpha)z}{1-z}+\frac{2(1-\alpha)z}{(1-z)^2}=:h(z)
$$ implies
$$\RE\left[(F'(z))^\mu\left(\frac{z}{F(z)}\right)^\nu\right]>\alpha.$$
Let $z=e^{i\theta},-\pi\leq\theta\leq\pi$. Then
\begin{eqnarray}
\nonumber
  \RE (h(e^{i\theta})) &=& \RE\set{(2\mu-\nu)\frac{1+(1-2\alpha)e^{i\theta}}{1-e^{i\theta}}+\frac{2(1-\alpha)e^{i\theta}}
  {(1-e^{i\theta})^2}} \\ \nonumber
   &=&(2\mu-\nu)\alpha- \frac{(1-\alpha)}{2}\left(\frac{1}{\sin^2{(\theta/2)}}\right)
   =: k(\theta).\\ \nonumber
\end{eqnarray}
A calculation shows that $k(\theta)$ attains its maximum at
$\theta=\pi$ and
$$\max_{|\theta|\leq \pi} k(\theta)= \frac{2(2\mu-\nu)\alpha-(1-\alpha)}{2}.$$
Hence the result follows at once.
\end{proof}
Taking $\mu=\nu=1$ in the Corollary~\ref{ex3}, we have the following
example.
\begin{example} Let $-1<B<A\leq1$. If $f\in \mathcal{A}$, $F$ as defined in (\ref{e2.12}) and
$$\frac{zF'(z)}{F(z)}\left(2\frac{f'(z)}{F'(z)}-\frac{f(z)}{F(z)}\right)
\prec\frac{1+Az}{1+Bz}+\frac{(A-B)z}{(1+Bz)^2},$$ then $F\in S^*[A,B].$
\end{example}
Putting $\mu=\nu=1$ in the Corollary~\ref{ex3.2}, we obtain the
following example.
\begin{example} \label{ex4} Let $0\leq\alpha<1$. If $f\in \mathcal{A},$ $F$ as defined in (\ref{e2.12}) and satisfies
$$\RE\set{\frac{zF'(z)}{F(z)}\left(2\frac{f'(z)}{F'(z)}-\frac{f(z)}{F(z)}\right)}<
\frac{(3\alpha-1)}{2},$$ then $F\in S^*(\alpha)$.
\end{example}
Putting $\mu=\nu=-1$ and assuming $f\in S^*$ in
Corollary~\ref{ex3.2}, we get the following example.
\begin{example}\label{exro2}  Let $0\leq\alpha<1$. If $f\in S^*,$ $F$ as defined in (\ref{e2.12}) and
$$\RE\set{\frac{F(z)}{zF'(z)}\left(\frac{f(z)}{F(z)}-2\frac{f'(z)}{F'(z)}\right)}<
-\frac{(\alpha+1)}{2},$$ then $F\in S_r^*(\alpha)$.
\end{example}
Putting $\mu=1$ and $\nu=0$ in Corollary~\ref{ex3.2}, we obtain the
following example.
\begin{example}\label{u3} Let $0\leq\alpha<1$. If $f\in \mathcal{A},$ $F$ as defined in (\ref{e2.12}) and
$$\RE f'(z)< \frac{5\alpha-1}{4},$$ then
$\RE F'(z)>\alpha.$
\end{example}
Putting $\mu=0$ and $\nu=-1$ in Corollary~\ref{ex3.2}, we have the
following example.
\begin{example}\label{exr} Let $0\leq\alpha<1$. If $f\in \mathcal{A},$ $F$ as defined in (\ref{e2.12}) and $$\RE\frac{f(z)}{z}<\frac{3\alpha-1}{2},$$ then
$F\in R(\alpha).$
\end{example}
By taking $\psi(z)=((1+z)/(1-z))^\eta$ in the subordination part of
Theorem~\ref{t2} for the case $L=H$, the Dzoik Srivastava operator,
we have the following result.
\begin{corollary}\label{ex3.1} Let $0<\eta\leq1$ and $\RE [(\mu-\nu)\alpha_1+\mu]\geq0$.  If $f\in \mathcal{A}_p,$ $F$ as defined in (\ref{e2.12}) and satisfies
the subordination
\begin{multline*}
\Omega_{H,\mu,\nu}^{\alpha_1}
(F(z))\left((\alpha_1+1)\mu\Omega_{H,1,0}^{\alpha_1}(f(z),F(z))-\nu\alpha_1\Omega_{H,0,-1}^{\alpha_1}(f(z),F(z))\right)\\\prec
\left(((\mu-\nu)\alpha_1+\mu)+\frac{2\eta
z}{(1-z^2)}\right)\left(\frac{1+z}{1-z}\right)^\eta,
\end{multline*}
 then
$$\Omega_{H,\mu,\nu}^{\alpha_1}(F(z))\prec\left(\frac{1+z}{1-z}\right)^\eta$$ and $((1+z)/(1-z))^\eta$ is the best dominant.
\end{corollary}
By putting $p=1, l=m+1,\alpha_1=1$ and $\alpha_{i+1}=\beta_i\;
(i=1,2,...m)$ in the above Corollary~\ref{ex3.1}, we obtain the
following result.
\begin{corollary}\label{ex3.21} Let $0<\eta\leq1$ and $2\mu\geq\nu$. If $f\in \mathcal{A},$ $F$ as defined in (\ref{e2.12}) and $$\left|\arg\set{(F'(z))^\mu\left(\frac{z}{F(z)}\right)^\nu\left(2\mu\frac{f'(z)}{F'(z)}-\nu\frac{f(z)}{F(z)}\right)}
\right|<\frac{(\eta+1)\pi}{2}-\arctan\frac{(2\mu-\nu)}{\eta},$$ then
$$\left|\arg\set{(F'(z))^\mu\left(\frac{z}{F(z)}\right)^\nu}\right|<\frac{\eta \pi}{2}.$$
\end{corollary}
\begin{proof} The proof of the above Corollary~\ref{ex3.21} is similar to that of the Corollary~\ref{ex1.1}
hence skipped here.

\end{proof}
In the above Corollary~\ref{ex3.21}, if we set $\mu=\nu=1$, then we
have the following example.
\begin{example}
Let $0<\eta\leq1$. If $f\in \mathcal{A},$ $F$ as defined in
(\ref{e2.12}) and
$$\left|\arg\set{\frac{zF'(z)}{F(z)}\left(2\frac{f'(z)}{F'(z)}-\frac{f(z)}{F(z)}\right)}\right|
<\frac{(\eta+1)\pi}{2}-\arctan\left(\frac{1}{\eta}\right),$$ then
$F(z)\in SS^*(\eta)$.
\end{example}
Taking the dominant $\psi(z)=\sqrt{1+z}$, which is a convex function
in the open unit disc $\mathbb{U}$, in the subordination part of
Theorem~\ref{t2}, we have the following corollary for the operator
$L=H$, the Dzoik Srivastava operator.
\begin{corollary}\label{lm1.4}  Let $0<\eta\leq1$ and $\RE [(\alpha_1(\mu-\nu)+\mu]\geq0$.  If $f\in \mathcal{A}_p,$ $F$ as defined in (\ref{e2.12}) and
\begin{multline*}
\Omega_{H,\mu,\nu}^{\alpha_1}
(F(z))\left((\alpha_1+1)\mu\Omega_{H,1,0}^{\alpha_1}(f(z),F(z))-\alpha_1\nu\Omega_{H,0,-1}^{\alpha_1}
(f(z),F(z))\right)\prec \\
(\alpha_1(\mu-\nu)+\mu)\sqrt{1+z}+\frac{z}{2\sqrt{1+z}},
\end{multline*}
 then
$$\Omega_{H,\mu,\nu}^{\alpha_1}(F(z))\prec\sqrt{1+z}$$ and $\sqrt{1+z}$ is the best dominant.
\end{corollary}
Putting  $p=1, l=m+1,\alpha_1=1$ and $\alpha_{i+1}=\beta_i\;
(i=1,2,...m)$ in Corollary~\ref{lm1.4}, we obtain the following
result.
\begin{corollary}\label{lm1.3} Let $0<\eta\leq1$ and $2\mu\geq\nu$. If $f\in \mathcal{A},$ $F$ as defined in (\ref{e2.12}) and
$$(F'(z))^\mu\left(\frac{z}{F(z)}\right)^\nu\left(2\mu\frac{f'(z)}{F'(z)}
-\nu\frac{f(z)}{F(z)}\right)\prec(2\mu-\nu)\sqrt{1+z}+\frac{z}{2\sqrt{1+z}},$$
then
$$(F'(z))^\mu\left(\frac{z}{F(z)}\right)^\nu\prec\sqrt{1+z}$$ and $\sqrt{1+z}$ is the best dominant.
\end{corollary}
Putting $\mu=\nu=1$ in the above Corollary~\ref{lm1.3}, we have the
following example.
\begin{example}  Let $0<\eta\leq1$. If $f\in \mathcal{A},$ $F$ as defined in (\ref{e2.12}) and
$$\left|\frac{zF'(z)}{F(z)}\left(2\frac{f'(z)}{F'(z)}
-\frac{f(z)}{F(z)}\right)\right|<\sqrt{1.22}\approx 1.10,$$ then
$F\in\mathcal{SL}.$
\end{example}
\begin{proof} The above result can be proved using the technique adopted in
the proof of the Example~\ref{elm1.1} and hence it is omitted here.
\end{proof}
Now we discuss some applications of Theorem~\ref{t1} when $L=I$, the
Integral transform. The subordination part of Theorem~\ref{t1}
yields the following corollary by taking $\psi(z)=(1+Az)/(1+Bz)$ and
$$\RE \left(1+\frac{z\psi''(z)}{\psi'(z)}\right)>\max\{0, \RE[(\nu-\mu)(\lambda+p)]\}$$ instead of taking ``$\psi$ is convex and $\RE[{(\mu-\nu)(\lambda+p)}]\geq0.$"
\begin{corollary}\label{cor3} Let $-1<B<A\leq1$ and $\lambda\neq-p$ be a complex number. Let $\RE (u-vB)\geq |v-\bar{u}B|$ where $u=(\mu-\nu)(\lambda+p)+1$, $v=\sqb{(\mu-\nu)(\lambda+p)-1}B$. If $f\in \mathcal{A}_p$ and
$$\Omega_{I,\mu,\nu}^r(f(z))\left(\mu\Omega_{I,1,1}^{r+1}(f(z))- \nu\Omega_{I,1,1}^r(f(z))\right)
\prec(\mu-\nu)\frac{1+Az}{1+Bz}+\frac{1}{\lambda+p}\frac{(A-B)z}{(1+Bz)^2},$$
then
$$\Omega_{I,\mu,\nu}^r(f(z))\prec\frac{1+Az}{1+Bz}$$ and $(1+Az)/(1+Bz)$ is the best dominant.
\end{corollary}
\begin{corollary}\label{cr3} Let $0\leq\alpha<1$, $\lambda\neq-p$ be a complex number and $\RE[(\mu-\nu)(\lambda+p)]\geq0.$ If $f\in \mathcal{A}_p$ and
$$\Omega_{I,\mu,\nu}^r(f(z))\left(\mu\Omega_{I,1,1}^{r+1}(f(z))- \nu\Omega_{I,1,1}^r(f(z))\right)
\prec(\mu-\nu)\frac{1+(1-2\alpha)z}{1-z}+\frac{1}{\lambda+p}\frac{2(1-\alpha)z}{(1-z)^2},$$
then
$$\Omega_{I,\mu,\nu}^r(f(z))\prec\frac{1+(1-2\alpha)z}{1-z}$$ and $(1+(1-2\alpha)z)/(1-z)$ is the best dominant.
\end{corollary}
Note that for $p=1,\lambda=0$ and $r=0$, we have $I_1(0,0)f(z)=f(z),
I_1(1,0)f(z)=zf'(z), I_1(2,0)f(z)=z(zf''(z)+f'(z)).$ Putting these
values in Corollary~\ref{cor3}, we have the following result.
\begin{corollary}\label{ex5} Let $-1<B<A\leq1$. Let $(u-vB)\geq |v-uB|$, where $u=\mu-\nu+1$ and $v=(\mu-\nu-1)B$. If $f\in \mathcal{A}$ and
$$(f'(z))^\mu\left(\frac{z}{f(z)}\right)^\nu\left(\mu\left(1+\frac{zf''(z)}{f'(z)}\right)-\nu\frac{zf'(z)}{f(z)}\right)
\prec(\mu-\nu)\frac{1+Az}{1+Bz}+\frac{(A-B)z}{(1+Bz)^2},$$ then
$$\left(f'(z)\right)^\mu \left(\frac{z}{f(z)}\right)^\nu \prec\frac{1+Az}{1+Bz}$$
and $(1+Az)/(1+Bz)$ is the best dominant.
\end{corollary}
\begin{corollary}\label{c1.12} Let $0\leq\alpha<1$ and $\mu\geq\nu$. If $f\in \mathcal{A}$ and satisfies
$$\RE\left[(f'(z))^\mu\left(\frac{z}{f(z)}\right)^\nu\left(\mu\left(1+\frac{zf''(z)}{f'(z)}\right)-\nu\frac{zf'(z)}
{f(z)}\right)\right]>\frac{2(\mu-\nu)\alpha-(1-\alpha)}{2},$$ then
$$\RE\left[\left(f'(z)\right)^\mu \left(\frac{z}{f(z)}\right)^\nu\right]>\alpha.$$
\end{corollary}
\begin{proof} The proof is similar to that of the Corollary~\ref{c1.11} hence omitted here.
\end{proof}
Putting $\mu=\nu=1$ in Corollary~\ref{ex5}, we have the following
result.
\begin{example} Let $-1<B<A\leq1$. If $f\in \mathcal{A}$ and satisfies
$$\frac{zf'(z)}{f(z)}\left(\left(1+\frac{zf''(z)}{f'(z)}\right)-\frac{zf'(z)}{f(z)}\right)
\prec\frac{(A-B)z}{(1+Bz)^2},$$ then $f\in S^*[A,B].$
\end{example}
Setting $\mu=\nu=1$ in Corollary~\ref{c1.12}, we have the following
result:
\begin{example}\label{ex1.2} Let $0\leq\alpha<1$. If $f\in \mathcal{A}$ satisfies the differential subordination
$$\RE\left[\frac{zf'(z)}{f(z)}\left(1-\frac{zf'(z)}{f(z)}+\frac{zf''(z)}{f'(z)}\right)\right]
>\frac{\alpha-1}{2},$$
then $f\in S^*(\alpha).$
\end{example}
\begin{rem} In fact for $\alpha=0$ the above Example~\ref{ex1.2} reduces to the
result {\cite[Corollary 2]{owa2}} due to Owa\ and Obradovi\'c.
\end{rem}
Putting $\mu=1$ and $\nu=0$ in Corollary~\ref{c1.12}, we have the
following result.
\begin{example}\label{ex11.1} Let $0\leq\alpha<1$. If $f\in \mathcal{A}$ and satisfies
$$\RE[ f'(z)+zf''(z)]>\frac{3\alpha-1}{2},$$ then
$\RE f'(z)>\alpha$.
\end{example}
\begin{rem}1. The above Example~\ref{ex11.1} extends the result  {\cite[Theorem 5]{chi}} due to Chichra. \\
2. Corollary~\ref{c1.12} reduces to {\cite[Theorem 2]{obra}} when
$\mu=0, \nu=-1$ and $\alpha=1/3.$
\end{rem}
If we take $\psi(z)=((1+z)/(1-z))^\eta$ with $0<\eta\leq1$, for the
case $L=I$, then clearly $\psi(z)$ is convex in the open unit disc
$\mathbb{U}$ and we have the following corollary from the
subordination part of Theorem~\ref{t1}.
\begin{corollary}\label{cr2} Let $0<\eta\leq1$, $\lambda\neq-p$ be a complex number and $\RE [(\mu-\nu)(\lambda+p)]\geq0$.  If $f\in \mathcal{A}_p,$ and satisfies the subordination
  $$\Omega_{I,\mu,\nu}^r(f(z))\left(\mu\Omega_{I,1,1}^{r+1}(f(z))- \nu\Omega_{I,1,1}^r(f(z))\right)\prec
\left((\mu-\nu)+\frac{2\eta
z}{(\lambda+p)(1-z^2)}\right)\left(\frac{1+z}{1-z}\right)^\eta,$$
then
$$\Omega_{I,\mu,\nu}^r(f(z))\prec\left(\frac{1+z}{1-z}\right)^\eta$$ and $\left(\frac{1+z}{1-z}\right)^\eta$ is the best dominant.
\end{corollary}
Putting $p=1,\lambda=0$ and $r=0$ in Corollary~\ref{cr2}, we obtain
the following corollary.
\begin{corollary}\label{cr1} Let $0<\eta\leq1$ and $\mu\geq\nu.$ If $f\in\mathcal{A}$ and satisfies
$$\left|\arg\set{(f'(z))^\mu\left(\frac{z}{f(z)}\right)^\nu\left(\mu\left(1+\frac{zf''(z)}{f'(z)}\right)-
\nu\frac{zf'(z)}{f(z)}\right)}\right|<\frac{\delta\pi}{2},$$ where
$$\delta=\eta+1-\frac{2}{\pi}\arctan\frac{\mu-\nu}{\eta},$$
then
$$\left|\arg\set{(f'(z))^\mu\left(\frac{z}{f(z)}\right)^\nu}\right|<\frac{\eta\pi}{2}.$$
\end{corollary}
\begin{proof} The proof of the above Corollary~\ref{cr1} is much akin to the proof of
 Corollary~\ref{ex1.1} hence it is left here.
\end{proof}
The following example is obtained by taking $\mu=\nu=1$ in the above
Corollary~\ref{cr1}.
\begin{example} If $f\in\mathcal{A}$ and satisfies
  $$\left|\arg\set{\frac{zf'(z)}{f(z)}\left(1+\frac{zf''(z)}{f'(z)}
-\frac{zf'(z)}{f(z)}\right)}\right|<\frac{(\eta+1)\pi}{2},$$ then
$f\in SS^*(\eta).$
\end{example}
Taking $\psi(z)=\sqrt{1+z}$, convex function in the open unit disc
$\mathbb{U}$, as dominant in the subordination part of the
Theorem~\ref{t1}, we obtain the following corollary.
\begin{corollary}\label{lm1.5}  Let $\lambda\neq-p$ be a complex number and $\RE [(\mu-\nu)(\lambda+p)]\geq0$.  If $f\in \mathcal{A}_p,$ and satisfies the subordination
 $$\Omega_{I,\mu,\nu}^r(f(z))\left(\mu\Omega_{I,1,1}^{r+1}(f(z))- \nu\Omega_{I,1,1}^r(f(z))\right)
 \prec(\mu-\nu)\sqrt{1+z}+\frac{z}{2(\lambda+p)\sqrt{1+z}},$$ then
 $$\Omega_{I,\mu,\nu}^r(f(z))\prec\sqrt{1+z}$$ and $\sqrt{1+z}$ is the best dominant.
\end{corollary}
Putting $p=1,\lambda=0$ and $r=0$ in Corollary~\ref{lm1.5}, we have
the following corollary.
\begin{corollary}\label{lm1.2} Let $\mu\geq\nu.$ If $f\in\mathcal{A}$ and satisfies the subordination
$$(f'(z))^\mu\left(\frac{z}{f(z)}\right)^\nu\left(\mu\left(1+\frac{zf''(z)}{f'(z)}\right)
-\nu\frac{zf'(z)}{f(z)}\right)\prec(\mu-\nu)\sqrt{1+z}+\frac{z}{2\sqrt{1+z}},$$
then $$(f'(z))^\mu\left(\frac{z}{f(z)}\right)^\nu\prec\sqrt{1+z}$$
and $\sqrt{1+z}$ is the best dominant.
\end{corollary}
\begin{example} If $f\in\mathcal{A}$ and satisfies
$$\left|\frac{zf'(z)}{f(z)}\left(1-\frac{zf'(z)}{f(z)}+\frac{zf''(z)}{f'(z)}\right)\right|
<\frac{1}{2\sqrt{2}}\approx 0.35,$$ then $f\in\mathcal{SL}.$
\end{example}
\begin{proof} Putting $\mu=\nu=1$ in Corollary~\ref{lm1.2} and using the technique used in the proof
of Example~\ref{elm1.1}, the proof follows at once.
\end{proof}
Now we will derive some applications of Theorem~\ref{t2} when $L=I$,
the Integral transform. The following corollary is obtained from the
subordination part of Theorem~\ref{t2}, by taking
$\psi(z)=(1+Az)/(1+Bz)$ and $$\RE
\left(1+\frac{z\psi''(z)}{\psi'(z)}\right)>\max\{0,\RE[(\nu-\mu)(\lambda+p)]\}$$
instead of taking ``$\psi$ is convex and
$\RE[(\mu-\nu)(\lambda+p)]\geq0".$
\begin{corollary}\label{cor4} Let $-1<B<A\leq1$ and $\RE (u-vB)\geq |v-\bar{u}B|$ where $u=(\mu-\nu)(\lambda+p)+1$, $v=[(\mu-\nu)(\lambda+p)-1]B$ and $\lambda\neq-p$ be a complex number. If $f\in \mathcal{A}_p$
and satisfies the subordination
\begin{multline*}
\Omega_{I,\mu,\nu}^r(f(z))\left(\mu\Omega_{I,1,0}^{r}(f(z),F(z))-
\nu\Omega_{I,0,-1}^r(f(z),F(z))\right)\\
\prec(\mu-\nu)\frac{1+Az}{1+Bz}+\frac{1}{\lambda+p}
\frac{(A-B)z}{(1+Bz)^2},
\end{multline*}
 then
$$\Omega_{I,\mu,\nu}^r(F(z))\prec\frac{1+Az}{1+Bz}$$ where $F$ is defined as in(\ref{e2.12}) and $(1+Az)/(1+Bz)$ is the best dominant.
\end{corollary}
Putting $p=1, \lambda=0$ and $r=0$, in the above
Corollary~\ref{cor4}, it reduces to Corollary~\ref{ex3}.
\begin{corollary}\label{cor4.1} Let $0\leq\alpha<1$, $\lambda\neq-p$ be a complex number and $\RE[(\mu-\nu)(\lambda+p)]\geq0$. If $f\in \mathcal{A}_p$ and satisfies the subordination
\begin{multline*}
\Omega_{I,\mu,\nu}^r(f(z))\left(\mu\Omega_{I,1,0}^{r}(f(z),F(z))-
\nu\Omega_{I,0,-1}^r(f(z),F(z))\right)\prec\\(\mu-\nu)\frac{1+(1-2\alpha)z}{1-z}+\frac{2(1-\alpha)}{\lambda+p}
\frac{z}{(1-z)^2},
\end{multline*} where $F$ is defined as in (\ref{e2.12}),
 then $$ \RE \Omega_{I,\mu,\nu}^r(f(z))>\alpha.$$
\end{corollary}
On setting $p=1, \lambda=0$ and $r=0$ the above
Corollary~\ref{cor4.1} reduces to Corollary~\ref{ex3.2}. Taking
$\psi(z)=((1+z)/(1-z))^\eta$, $0<\eta\leq1,$ as dominant in the
subordination part of
 Theorem~\ref{t2}, for the Integral operator $I=L$, we have the following corollary. Further on setting $p=1, \lambda=0$ and $r=0$, it reduces to Corollary~\ref{ex3.21}.
\begin{corollary}\label{cor4.2} Let $0<\eta\leq1$, $\lambda\neq-p$ be a complex number and $\RE [(\mu-\nu)(\lambda+p)]\geq0$. If $f\in \mathcal{A}_p,$ $F$ as defined in (\ref{e2.12})
and satisfies the subordination
\begin{multline*}
\Omega_{I,\mu,\nu}^r(f(z))\left(\mu\Omega_{I,1,0}^{r}(f(z),F(z))-
\nu\Omega_{I,0,-1}^r(f(z),F(z))\right)\\
\prec\left((\mu-\nu)+ \frac{2\eta
z}{(\lambda+p)(1-z^2)}\right)\left(\frac{1+z}{1-z}\right)^\eta,
\end{multline*}
then
$$\Omega_{I,\mu,\nu}^r(f(z))\prec\left(\frac{1+z}{1-z}\right)^\eta$$
and $((1+z)/(1-z))^\eta$ is the best dominant.
\end{corollary}
Taking $\psi=\sqrt{1+z}$ in the subordination part of
Theorem~\ref{t2}, we have the following corollary corresponding to
the integral operator $I=L$, which finally reduces to
Corollary~\ref{lm1.3}, when $p=1, \lambda=0$ and $r=0$.
\begin{corollary}  Let $\lambda\neq-p$ be a complex number and $\RE [(\mu-\nu)(\lambda+p)]\geq0$.  If $f\in \mathcal{A}_p,$ $F$ as defined in (\ref{e2.12}) and
\begin{multline*}
\Omega_{I,\mu,\nu}^r(f(z))\left(\mu\Omega_{I,1,0}^{r}(f(z),F(z))-
\nu\Omega_{I,0,-1}^r(f(z),F(z))\right)\\
\prec(\mu-\nu)\sqrt{1+z}+ \frac{z}{2(\lambda+p)\sqrt{1+z}},
\end{multline*}
where $F$ is defined as in (\ref{e2.12}), then
$$\Omega_{I,\mu,\nu}^r(f(z))\prec\sqrt{1+z}$$ and $\sqrt{1+z}$ is the best dominant.
\end{corollary}

\end{document}